\documentclass[a4paper,reqno,11pt]{amsart}
\usepackage{amsmath,amssymb,amsthm,mathtools,calc,verbatim,enumitem,tikz,url,mathrsfs,fullpage, ulem}
\usepackage[nocompress]{cite}
\usepackage{bbm}
\usepackage{stmaryrd}
\usepackage{textcomp}
\usepackage{setspace}
\usepackage{amssymb}
\usepackage{amsthm}
\usepackage{amsmath}
\usepackage{graphicx}
\usepackage{tikz}
\usetikzlibrary{calc}
\usetikzlibrary{arrows.meta, positioning}
\usepackage{mathrsfs}

\usepackage{algorithm}
\usepackage{algpseudocode}
\usepackage{float}

\usepackage{caption}
\usepackage{marvosym}
\usepackage{empheq}
\usepackage{latexsym}
\usepackage{fontenc}
\usepackage{color}
\usepackage{thmtools, thm-restate}

\newtheorem{theorem}{Theorem}[section]

\newtheorem{lemma}[theorem]{Lemma}
\newtheorem{corollary}[theorem]{Corollary}

\newtheorem{claim}[theorem]{Claim}
\newtheorem*{claim*}{Claim}

\theoremstyle{definition}
\newtheorem{definition}[theorem]{Definition}

\newtheorem{question}[theorem]{Question}
\newtheorem*{qu*}{Question}
\theoremstyle{remark}

\newcommand\cA{\mathcal{A}}
\newcommand\cB{\mathcal{B}}

\newcommand\cG{\mathcal{G}}
\newcommand\cH{\mathcal{H}}
\newcommand\cI{\mathcal{I}}

\newcommand\cR{\mathcal{R}}
\newcommand\cS{\mathcal{S}}

\renewcommand\Pr{\operatorname{\mathbb{P}}}

\renewcommand\leq{\leqslant}
\renewcommand\geq{\geqslant}
\renewcommand\le{\leqslant}
\renewcommand\ge{\geqslant}
\renewcommand\to{\rightarrow}

	\def\Prob{\mathbb{P}}

	\def\cA{\mathcal{A}}

	\def\cF{\mathcal{F}}

	\def\<{\langle }
	\def\>{\rangle }

\usepackage{soul}
\usepackage{hyperref}
\usepackage{cleveref}

\crefname{thm}{theorem}{theorems}
\crefname{lem}{lemma}{lemmas}
\crefname{obs}{observation}{observations}
\crefname{defi}{definition}{definitions}
\Crefname{thm}{Theorem}{Theorems}
\Crefname{prop}{Proposition}{Propositions}
\Crefname{obs}{Observation}{Observations}
\Crefname{claim}{Claim}{Claims}
\Crefname{defi}{Definition}{Definitions}
\Crefname{subsection}{Section}{Sections}
\Crefname{question}{Question}{Questions}

\crefformat{enumi}{#2item~#1#3}
\Crefformat{enumi}{#2Item~#1#3}
\crefmultiformat{enumi}
  {items~#2#1#3}
  { and~#2#1#3}
  {, #2#1#3}
  {, and~#2#1#3}
\crefrangeformat{enumi}
  {items~#3#1#4--#5#2#6}

\begin{document}

\normalem

\title{Ramsey properties for tilings in random graphs}
\author{Lucas Arag\~ao \and Xinbu Cheng \and Rafael Filipe \and Rafael Miyazaki\and \\Danni Peng \and Zhifei Yan}

\address{Departamento de Matemática Aplicada, Instituto de Matemática, Universidade Federal do Rio de Janeiro, Rio de Janeiro, 21941-909, Brasil}
\email{aragao@im.ufrj.br}

\address{IMPA, Estrada Dona Castorina 110, Jardim Bot\^anico, Rio de Janeiro, 22460-320, Brasil}
\email{\{rafael.santos, danni.peng, xinbu.cheng\}@impa.br}

\address{Department of Mathematics, Emory University, Atlanta, GA, 30322, USA}
\email{rafael.kazuhiro.miyazaki@emory.edu}

\address{ECOPRO, Institute for Basic Science, 55 Expo-ro, Yuseong-gu, Daejeon, 34126, Korea}\email{zhifeiyan@ibs.re.kr}

\thanks{During this work, Xinbu Cheng was supported by CAPES, Rafael Filipe and Danni Peng were supported by CNPq, and Zhifei Yan was supported by the Institute for Basic Science (IBS-R029-C4)}

\begin{abstract}
Let $mH$ be the graph formed by $m$ vertex-disjoint copies of a graph $H$. Let $G \to (H)_r$ denote that, in any $r$-colouring of the edges of $G$, there exists a monochromatic copy of $H$. In 1975, Burr, Erd\H{o}s, and Spencer showed that if $H$ is a graph on $k$ vertices whose independence number is $\alpha$, then $K_n \to (mH)_2$, where $m\sim n/(2k-\alpha)$, and that the $1/(2k-\alpha)$ factor is best possible. In the 1990s, R\"{o}dl and Ruci\'{n}ski proved that, for all but a few graphs~$H$, the threshold for the property $\mathbb{G}(n,p) \to (H)_r$ is $n^{-1/m_2(H)}$. In this paper, generalizing the result of Burr, Erd\H{o}s, and Spencer, we prove that $n^{-1/\max\{m_2(H),1\}}$ is the threshold for the property $\mathbb{G}(n,p) \to (mH)_2$, where $m\sim n/(2k-\alpha)$. This threshold matches the one found by R\"{o}dl and Ruci\'nski for most graphs $H$, extending their result in the case $r=2$.
\end{abstract}

\maketitle
\section{Introduction}

Ramsey theory is one of the central topics in combinatorics. A fundamental problem in this field is to determine the \emph{Ramsey number} of a given graph~$H$, denoted by $R(H)$. This is the smallest integer~$n$ such that every $2$-colouring of the edges of the complete graph~$K_n$ contains a monochromatic copy of~$H$. A classical theorem of Ramsey~\cite{Ramsey} asserts that $R(H)$ is finite for every graph~$H$, and in 1935 Erd\H{o}s and Szekeres~\cite{E-Sz} established the first explicit upper bound for Ramsey numbers. Later, in 1947, Erd\H{o}s~\cite{Ramsey.lower} gave an exponential lower bound for the case $H = K_k$, marking the first application of the probabilistic method. This lower bound has since only been improved by a constant factor~\cite{Spencer_Diagonal}. On the other hand, the upper bound in the case $H=K_k$ has seen several substantial improvements over the years~\cite{Thomason_Diagonal,Conlon_Diagonal,Sah_Diagonal}, although an exponential improvement was achieved only recently by Campos, Griffiths, Morris, and Sahasrabudhe~\cite{CGMS}. Together with Balister, Bollobás, Hurley, and Tiba~\cite{BBCGHMST}, they also obtained new upper bounds for multicolour Ramsey numbers. For further background on Ramsey theory, we refer the reader to the surveys~\cite{CFS, Morris2026}.

An interesting variant of the classical Ramsey problem concerns the study of Ramsey numbers for \emph{graph tilings}. Given a graph $H$, an $H$-\textit{tiling} is a collection of vertex-disjoint copies of $H$, and we write $mH$ to denote an $H$-tiling of size $m$. One of the earliest questions in this direction, posed by Erd\H{o}s~\cite{Erdos}, asked for estimates of $R(mK_t)$ for all $t \ge 3$. In 1975, Burr, Erd\H{o}s, and Spencer~\cite{BES} answered this question in a strong form by establishing the following theorem.

\begin{theorem}[Burr, Erd\H{o}s, and Spencer~\cite{BES}]
\label{thm:BESOriginal}
    Let $H$ be a graph on $k$ vertices that contains no isolated vertices and whose independence number is $\alpha$. There exist constants $c = c(H)$ and $m_0 \in \mathbb{N}$ such that
\[
R(mH) = (2k - \alpha)m + c
\]
for all $m \ge m_0$.
\end{theorem}

In their paper, Burr, Erd\H{o}s, and Spencer~\cite{BES} stated \Cref{thm:BESOriginal} in a more general setting, deriving bounds for $R(mH,nG)$ for arbitrary graphs $H$ and $G$ without isolated vertices. However, in the particular case of \Cref{thm:BESOriginal}, their method does not allow one to compute the constant $c$ explicitly in terms of $H$, nor does it provide bounds on $m_0$. A few years later, Burr~\cite{Burr} refined their results by determining the exact dependence of the constant $c$ on the graphs for the cases $R(G,mH)$ and $R(mK_r,mK_s)$. More recently, Buci\'{c} and Sudakov~\cite{BS} showed that $m_0 = 2^{O(k)}$ suffices, and determined the constant $c(H)$ in the case where $H$ is a clique.

The statement of \Cref{thm:BESOriginal} naturally leads to an inverse version of the Ramsey problem for graph tilings, which we formulate as follows.

\begin{question}\label{question1}
For $n \in \mathbb{N}$ and a given graph $H$, how large a monochromatic $H$-tiling is guaranteed in any $r$-colouring of the edges of $K_n$?
\end{question}

Given two graphs $H$ and $G$, we will write $Rt_r(H,G)$ for the largest integer $m$ such that in every $r$-colouring of the edges of $G$, there is a monochromatic $H$-tiling of size $m$ in~$G$, with all the copies of the same colour. When $r=2$, we simply write $Rt(H,G)$. Using this notation, we can restate the result of Burr, Erd\H{o}s, and Spencer~\cite{BES} for $R(nH)$ as follows.

\begin{theorem}[Burr, Erd\H{o}s, and Spencer~\cite{BES}]\label{thm:BES}
    Let $H$ be a graph on $k$ vertices that contains no isolated vertices and whose independence number is $\alpha$. There exist constants $n_0 \in \mathbb{N}$ and $c=c(H)$ such that
\begin{equation*}
Rt(H,K_n) = \left\lfloor{\frac{n}{2k - \alpha} -c}\right\rfloor    \end{equation*}
for all $n \ge n_0$.
\end{theorem}

A variant of this problem was considered by Balogh, Freschi, and Treglown~\cite{BFT} who studied the setting in which the host graph is not complete. In particular, they studied the case where the host graph $G$ has linear minimum degree. They determined the exact value of $Rt(K_3, G)$ when $4n/5\leq\delta(G) \le {5n}/{6}$ and when $\delta(G) \ge {65n}/{66}$. However, the value of $Rt(K_3, G)$ remains unknown in the intermediate range of $\delta(G)$. For general graphs $H$, the question remains wide open.

Another natural setting arises when the host graph is the random graph. Gishboliner, Krivelevich, and Michaeli~\cite{GKM} determined the threshold for the property that, in any $r$-colouring of the edges of $\mathbb{G}(n,p)$, there exists an almost monochromatic path of size $2n/(r+1)-o(n)$, possibly excluding a constant number of edges. Their result (Theorem~1.5 in~\cite{GKM}) immediately yields the following random version of \Cref{thm:BES} for perfect matchings.

\begin{theorem}[Gishboliner, Krivelevich, and Michaeli~\cite{GKM}]\label{thm:MatchingLower}
Let $r \ge 2$ be a positive integer. For every $\varepsilon>0$, there exists $C=C(\varepsilon)>0$ such that the following holds. If $p\geq C/n$ and $G\sim \mathbb{G}(n,p)$, then
\begin{align*}
 Rt_r(K_2,G)\geq \frac{n}{r+1}-\varepsilon n  
\end{align*}
with high probability.
\end{theorem}

This result was later extended to random hypergraphs by Gishboliner, Glock, Michaeli, and Sgueglia \cite{GGMS}.

In this paper, we investigate \Cref{question1} in the case $r=2$ when the host graph is the random graph~$\mathbb{G}(n,p)$, for general graphs $H$.

\subsection{Ramsey theory in random graphs}Studying various Ramsey properties in random graphs $\mathbb{G}(n,p)$ has been one of the major and most dynamic directions in combinatorics over the past few decades.

A landmark result relating Ramsey theory and random graphs was established by R\"{o}dl and Ruci\'{n}ski~\cite{RR93, RR94, RR95} in a series of papers in the 1990s. They determined the threshold for the appearance of a monochromatic copy of a general graph~$H$ in every \mbox{$r$-colouring} of the edges of $\mathbb{G}(n,p)$. They proved that the threshold in this case is determined by the $m_2$-\textit{density} of $H$. Given a graph $H$ with at least two edges, the \emph{$m_2$-density} of $H$ is defined as
\[m_2(H) \coloneq\max\left\{\frac{e(F)-1}{v(F)-2} \colon F \subseteq H, \, v(F) \ge 3 \right\}.\]
For all other graphs $H$, we define $m_2(H)=1/2$. Formally, they proved the following result for random graphs.

\begin{theorem}[R\"{o}dl and Ruci\'{n}ski \cite{RR93, RR94, RR95}]\label{thm:RR}
Let $r\ge2$ be an integer and $H$ be a graph on $k$ vertices which is not a forest of stars or, in the case $r=2$, paths of length 3. Then $n^{-1/m_2(H)}$ is the threshold for the property that, in any $r$-colouring of the edges of $\mathbb{G}(n,p)$, there is a monochromatic copy of $H$.
\end{theorem}

In this paper, we extend the classical result of R\"{o}dl and Ruci\'{n}ski in the two-colour case. More precisely, we show that, for a large class of graphs~$H$, the threshold for the appearance of a monochromatic copy of $H$ in every $2$-colouring of the edges of $\mathbb{G}(n,p)$ coincides with the threshold for the existence of a monochromatic $H$-tiling of asymptotically maximum possible size in every such colouring. The following theorem is our main result and provides such a statement.

\begin{theorem}\label{thm:Hlower}
Let $H$ be a graph on $k$ vertices whose independence number is $\alpha$. For every $\varepsilon>0$, there exists $C=C(H,\varepsilon)>0$ such that the following holds. If $p\geq Cn^{-1/\max\{m_2(H),1
\}}$ and $G\sim \mathbb{G}(n,p)$, then
\begin{align*}
 Rt(H,G)\geq \frac{n}{2k-\alpha}-\varepsilon n  
\end{align*}
with high probability.
\end{theorem}

We remark that $m_2(H) <1$ if and only if $H$ is a matching. In this case, \Cref{thm:Hlower} recovers the case $r=2$ of the result by Gishboliner, Krivelevich, and Michaeli \cite{GKM}, stated here as \Cref{thm:MatchingLower}. Combining \Cref{thm:Hlower} and \Cref{thm:BES} yields the following probabilistic generalization of \Cref{thm:BES}.

\begin{corollary}\label{cor:Hlower1}
    Let $H$ be a graph on $k$ vertices that contains no isolated vertices and whose independence number is $\alpha$. If $p\gg n^{-1/\max\{m_2(H),1\}}$ and $G\sim \mathbb{G}(n,p)$, then
\begin{align*}
 Rt(H,G) = \frac{n}{2k-\alpha}-o(n)
\end{align*}
with high probability. 
\end{corollary}

We also remark that \Cref{thm:Hlower} extends \Cref{thm:RR} in the case $r=2$ as follows. For all graphs~$H$ which are not a forest of stars and paths of length $3$ and $G \sim \mathbb{G}(n,p)$, if  $p \ll n^{-1/\max\{m_2(H),1\}}$, then there exists a $2$-colouring of $G$ with no monochromatic copies of $H$. On the other hand, if $p \gg n^{-1/\max\{m_2(H),1\}}$, then every $2$-colouring of $G$ yields not only a single monochromatic copy of $H$, but also a monochromatic $H$-tiling of size $n/(2k-\alpha)-o(n)$, which is asymptotically optimal. Indeed, for these graphs $H$, we can simply combine \Cref{thm:RR,thm:Hlower}.

If $H$ is a forest of stars or paths of length $3$, then $m_2(H) \le 1$, and \Cref{thm:Hlower} yields the same $1$-statement above. However, although we cannot guarantee that $G$ does not contain a monochromatic copy of $H$ for $p=o(n^{-1})$, we guarantee that $G$ does not contain an $H$-tiling of size $n/(2k-\alpha)-o(n)$. Indeed, observe that selecting one edge from each copy of $H$ in an $H$-tiling produces a matching in $G$. Hence, if $p = o(n^{-1})$, Markov's inequality implies that with high probability $\mathbb{G}(n,p)$ does not contain a matching of size $\Omega(n)$ and thus contains no $H$-tilings (monochromatic or otherwise) of size $n/(2k-\alpha)-o(n)$.

The proof of \Cref{thm:Hlower} is inspired by the approach of Burr, Erd\H{o}s, and Spencer~\cite{BES}. To obtain a large collection of monochromatic vertex-disjoint copies of $H$ in any $2$-colouring of the edges of $\mathbb{G}(n,p)$, we search for structures that contain many monochromatic copies of $H$ while covering relatively few vertices. The main difficulty appears when $H$ is not a clique, since in this case the \textit{bow tie} like structures considered by Burr, Erd\H{o}s, and Spencer~\cite{BES} are hard to find in a $2$-colouring of the edges of $\mathbb{G}(n,p)$. To address this difficulty, for a graph $H$ and a parameter $\eta$, we introduce a new structure that we call $(H, \eta)$-\textit{clusters}, see \Cref{def:clusters}. Considering these structures is the key new idea in the proof for general $H$. Our argument relies on the method of hypergraph containers, which allows us to establish a structural pseudorandom property satisfied by all vertex subsets of size $\varepsilon n$ in $\mathbb{G}(n,p)$, under any $2$-colouring of its edges. This property, in turn, enables us to identify such $(H, \eta)$-clusters (see \Cref{subsec:2.3}).

\subsection{The method of hypergraph containers} The hypergraph container method, developed by Balogh, Morris, and Samotij \cite{BMS} and independently by Saxton and Thomason \cite{ST}, is a powerful tool for controlling the probability that a random set avoids some forbidden substructure. In general, the basic container lemma asserts that the sets avoiding these substructures are ``clustered", in the sense that they are contained in a small collection of vertex subsets $\mathcal{C}$ (the containers), and moreover each $C\in\mathcal{C}$ spans only a few copies of the forbidden structures.

The container method has proven highly effective in a wide range of problems in Ramsey theory for random graphs. Of particular relevance to this work, it was used by Nenadov and Steger~\cite{NS} to provide a concise proof of the 1-statement in \Cref{thm:RR}, and our argument closely follows their framework. We first establish a container lemma, \Cref{lem:ABcontainer}, and then apply it to prove \Cref{thm:Htrans}, showing that $\mathbb{G}(n,p)$ satisfies a certain property with high probability.

In other applications, Schacht and Schulenburg~\cite{SS} established the sharpness of the threshold for Ramsey properties in certain classes of graphs, while R\"{o}dl, Ruci\'{n}ski, and Schacht~\cite{RRS} derived a quantitative version of~\Cref{thm:RR}, and Mousset, Nenadov, and Samotij~\cite{MNS} obtained an upper bound on the threshold function for asymmetric Ramsey numbers. This last result establishes the $1$-statement of a well-known conjecture of Kohayakawa and Kreuter, as formulated in~\cite{KSS}. The $0$-statement of this conjecture was recently proved by Christoph, Martinsson, Steiner, and Wigderson~\cite{CMSW}. 

More recently, Friedgut, Kuperwasser, Samotij, and Schacht~\cite{FKSS} developed a unified framework for establishing sharp threshold results for a broad class of Ramsey properties, while Alvarado, Kohayakawa, Morris, and Mota~\cite{AKMM,AKMM2} investigated canonical Ramsey properties of random graphs. Finally, in a recent breakthrough, Aragão, Campos, Dahia, Filipe, and Marciano~\cite{ACDFM} developed a new container theorem yielding an exponential upper bound on induced Ramsey numbers, thereby resolving a conjecture of Erd\H{o}s from 1975.

\subsection{Organization of the paper and notation} 

In what follows, we omit floor and ceiling symbols and assume divisibilities whenever the case distinction does not affect the argument.

The remainder of this paper is organized as follows. In \Cref{sec:2}, we outline the main ideas of the proof, beginning with the simplest case of monochromatic triangle tilings in $\mathbb{G}(n,p)$, which serves as a warm-up and illustrates the key techniques used throughout. We then provide an overview of the general case and prove \Cref{thm:Hlower}, assuming \Cref{thm:Htrans}, a property of $\mathbb{G}(n,p)$ in the relevant regime, and \Cref{lem:generalHtie}, which forms the core of our proof. We also prove \Cref{thm:Htrans} under the assumption of \Cref{lem:ABcontainer}, our container lemma.

In \Cref{sec:3}, we prove \Cref{lem:generalHtie}. In \Cref{sec:4}, we prove \Cref{lem:ABsuper}, the supersaturation result required for our application of the container method. In \Cref{sec:5}, we combine \Cref{lem:ABsuper} with the hypergraph container method to prove \Cref{lem:ABcontainer}.

\section{Outline of the proof}\label{sec:2}

In this section, we outline the proof of \Cref{thm:Hlower}, introducing the main definitions and results. We first consider the case of monochromatic $K_3$-tilings in an arbitrary $2$-colouring of edges of a random graph, which serves to illustrate the general framework of the proof. Subsequently, we explain how to adapt the ideas from this particular setting to prove the general case. Finally, we describe the role played by the hypergraph container method in the proof of \Cref{thm:Hlower}.

\subsection{$K_3$-tilings in random graphs} We begin studying the case of monochromatic $K_3$-tilings under an arbitrary $2$-colouring of the edges of a random graph. Precisely, we prove the following theorem, which corresponds to \Cref{thm:Hlower} for $H = K_3$.

\begin{theorem}\label{thm:trianglelower}
For every $\varepsilon>0$, there exists $C>0$ such that the following holds. If $p \geq Cn^{-1/2}$ and $G\sim \mathbb{G}(n,p)$, then
\begin{equation*}
    Rt(K_3,G)\geq \frac{n}{5}-\varepsilon n
\end{equation*}
with high probability.
\end{theorem}

Before proceeding to the proof of \Cref{thm:trianglelower}, we first provide an overview of the proof. Motivated by an idea of Burr, Erd\H{o}s, and Spencer~\cite{BES}, in order to construct a large monochromatic $K_3$-tiling in a $2$-colouring of the edges of $G \sim \mathbb{G}(n,p)$ we first select a maximal family $\mathcal{T}$ of disjoint \emph{bow ties}, where by a bow tie we mean a subgraph consisting of two monochromatic triangles of different colours that share a common vertex. The benefit of considering this type of structure is that a bow tie contains a monochromatic $K_3$ in each colour, while covering only five vertices.

If this family of bow ties spans a large proportion of the vertices of the graph, we will be able to find a large enough $K_3$-tiling in both colours. If this is not the case, the next step is to show that in the remaining vertices of $\mathbb{G}(n,p)$ we can find, in one of the two colours, a $K_3$-tiling that covers almost all of the remaining vertices.

To this end, the following lemma, applied iteratively to the remaining vertices of $\mathbb{G}(n,p)$, yields many disjoint monochromatic copies of $K_3$, not necessarily all of the same colour. We observe that \Cref{lem:trianglesup} can be derived from the results of R\"{o}dl and Ruci\'{n}ski \cite{RR93, RR94, RR95} and also follows from \Cref{thm:triangletrans}, presented below.

\begin{lemma}\label{lem:trianglesup}
For every $\varepsilon>0$, there exists $C>0$ such that the following holds. If $p \geq Cn^{-1/2}$ and $G\sim \mathbb{G}(n,p)$, then with high probability, for every $2$-colouring of the edges of $G$ and every $U\subset V(G)$ with $|U|\geq \varepsilon n$, there exists a monochromatic triangle in $G[U]$.
\end{lemma}

Let $\cB$ and $\cR$ be the families of disjoint blue and red triangles, respectively, obtained by iteratively applying \Cref{lem:trianglesup}. Recall that we set aside $\mathcal{T}$, a maximal family  of bow ties. If one of these two families, together with the bow ties we set aside, provides enough disjoint triangles of a single colour, we find the large monochromatic $K_3$-tiling required by \Cref{thm:trianglelower}. Otherwise, $\cB$ and $\cR$ are both of size $\Theta(n)$. The next step is to show that this cannot happen. Indeed, in this case, we will show that one can find a new monochromatic triangle intersecting these disjoint families, which yields a new bow tie, contradicting the maximality of~$\mathcal{T}$.

The difficulty here is to construct this new bow tie using the structural properties given by families $\cB$ and $\cR$. To deal with this, we prove that $\mathbb{G}(n,p)$ satisfies a certain property with high probability. To make this precise, we first introduce the following definition.

\begin{definition}\label{def:K3rich}
Let $s \in \mathbb{N}$. We say that a graph $G$ is $(K_3,s)$-\textit{rich} if for every $2$-colouring of the edges of $G$ and every pair of disjoint $s$-sets $X, Y \subset V(G)$, at least one of the following holds:

\begin{enumerate}[label =\textnormal{(\roman*)}]
\item\label{item:2.3.i} there exists a red triangle $T$ in $G[X\cup Y]$ such that $|V(T) \cap X| \ge 1$;
\item\label{item:2.3.ii} there exists a blue triangle $T$ in $G[X\cup Y]$ such that $|V(T) \cap Y| \ge 1$.
\end{enumerate}
\end{definition}

We now describe how to find a bow tie assuming this property. If each of the families $\cB$ and $\cR$ spans at least $s$ vertices and $\mathbb{G}(n,p)$ is $(K_3,s)$-rich, then there exists either a red triangle intersecting the vertex set spanned by $\cB$ or a blue triangle intersecting the vertex set spanned by $\cR$. Without loss of generality, let $T$ be a red triangle that intersects the vertex set spanned by $\cB$, and let $T'$ be a blue triangle in $\cB$ that intersects $T$. Since $T$ and $T'$ have different colours, they intersect in exactly one vertex, and hence $T \cup T'$ is a bow tie.
    
The following theorem shows that, in the range of $p$ that we are interested in, the random graph $\mathbb{G}(n,p)$ is $(K_3,s)$-rich with high probability for $s$ linear in $n$.

\begin{theorem}\label{thm:triangletrans}
For every $\varepsilon > 0$, there exists $C>0$ such that the following holds. If $p\geq Cn^{-1/2}$ and $G\sim \mathbb{G}(n,p)$, then with high probability, $G$ is $(K_3,\varepsilon n)$-rich.
\end{theorem}

We remark that \Cref{lem:trianglesup} and \Cref{thm:triangletrans} correspond to the special case $H=K_3$ of \Cref{lem:Hsup} and \Cref{thm:Htrans}, respectively, presented below. We discuss the proof of these generalizations when treating the general case. 

\begin{proof}[Proof of \Cref{thm:trianglelower} assuming \Cref{thm:triangletrans}] Given $\varepsilon>0$, let $C$ be sufficiently large and let $G \sim \mathbb{G}(n,p)$ for some $p \ge Cn^{-1/2}$. By \Cref{lem:trianglesup} and \Cref{thm:triangletrans}, with high probability $G$ is $(K_3,4\varepsilon n)$-rich and for every $2$-colouring of the edges of $G$ and every $U\subset V(G)$ with $|U|\geq \varepsilon n$, there exists a monochromatic triangle in $G[U]$. Our goal is to show that for such graphs $G$, every $2$-colouring of the edges of $G$ contains a monochromatic $K_3$-tiling with at least $n/5-\varepsilon n$ triangles.

For a fixed $2$-colouring of $G$, let $\mathcal{T}$ be a maximal collection of vertex-disjoint bow ties in $G$, and let $T = T_1\cup T_2\cup\dots\cup T_t$ be the vertex sets of the bow ties in $\mathcal{T}$. Suppose that $t < n/5 - \varepsilon n$, since otherwise we are done. By \Cref{lem:trianglesup}, we can find $(n - 5t - \varepsilon n)/3$ disjoint monochromatic triangles in $V(G)\setminus T$. Note that all but at most $4\varepsilon n/3$ of these triangles must have the same colour. Indeed, if there are at least $4\varepsilon n/3$ triangles of each colour, let $B$ and $R$ be the sets of vertices spanned by $4\varepsilon n/3$ blue and $4\varepsilon n/3$ red triangles, respectively. Since $G$ is $(K_3, 4\varepsilon n)$-rich, the sets $B$ and $R$ will provide a new bow tie that extends the collection of bow ties, contradicting the maximality of $\mathcal{T}$. Hence, we can find a monochromatic triangle tiling with at least
\[t + \frac{n - 5t - \varepsilon n}{3}-\frac{4\varepsilon n}{3}>\frac{n}{5}-\varepsilon n\]
copies of $K_3$, where the inequality is true since $t<n/5-\varepsilon n$. This completes the proof.
\end{proof}

\subsection{General $H$-tiling in random graphs}\label{subsec:2.2}
Analogously to the $K_3$ case, to construct a large monochromatic $H$-tiling in a $2$-colouring of the edges of $G \sim \mathbb{G}(n,p)$, a natural first approach is to consider bow-tie-like structures consisting of two monochromatic copies of $H$ in different colours whose union covers a small number of vertices. One might then try to construct such structures by taking two monochromatic copies of $H$ in different colours intersecting in a maximum independent set of~$H$. As first introduced in \cite{BES}, such coloured graphs are called $H$-\emph{ties}. Finding a large family of vertex-disjoint $H$-ties immediately yields a large monochromatic $H$-tiling. This is the main idea of the construction for~$K_n$ found in \cite{BES}.

By adapting the proof of the triangle case to the general setting, we can similarly find families $\cB$ and $\cR$ of disjoint monochromatic copies of $H$ in blue and red, respectively, each of size $\Theta(n)$. Let $B$ be the vertex set spanned by $\cB$ and let $R$ be the vertex set spanned by $\cR$. We then seek to find an $H$-tie using the structure of $\cB$ and $\cR$, by finding either a red copy of $H$ that intersects $B$ in at least $\alpha(H)$ vertices or a blue copy of $H$ that intersects $R$ in at least $\alpha(H)$ vertices. The difficulty here is that, unlike the triangle case, finding such a monochromatic copy of $H$ does not guarantee the existence of an $H$-tie. For instance, if the new copy is red and intersects $B$ in at least $\alpha(H)$ vertices, these vertices may lie in several different blue copies of $H$ in $B$. Observe that this problem does not occur when $H$ is a clique, since in this case $\alpha(H)=1$. In fact, we can prove \Cref{thm:Hlower} when $H$ is a clique by following the same steps as in the proof of \Cref{thm:trianglelower}.

To address the difficulty discussed above for general $H$, rather than looking for individual copies of $H$-ties that yield only one red and one blue copy of $H$, we introduce the following broader and more flexible structure.

\begin{definition}\label{def:clusters} Let $H$ be a graph on $k$ vertices whose independence number is $\alpha$, and $\eta\geq 0$ be a constant. A nonempty $2$-coloured graph $F$ is called an $(H,\eta)$-\emph{cluster} if it contains both a red and a blue $H$-tiling, each with at least $v(F)/(2k-\alpha)-\eta v(F)$ copies of $H$.
\end{definition}

Unlike the rigid structure of ordinary $H$-ties, these $(H,\eta)$-clusters are defined solely by their capacity to produce comparable amounts of red and blue copies of $H$ while using relatively few vertices. Luckily, an extension of \Cref{def:K3rich} for a general graph $H$ is still useful to find $(H, \eta)$-clusters.

\begin{definition}\label{def:Hrich}
Let $H$ be a graph and let $s$ be a positive integer. We say that a graph $G$ is $(H,s)$-\textit{rich} if for every $2$-colouring of the edges of $G$ and every pair of disjoint $s$-sets $X, Y \subset V(G)$, at least one of the following holds:
\begin{enumerate}[label =\textnormal{(\roman*)}]
\item\label{item:3.4.i} there exists a red copy $H'$ of $H$ in $G[X\cup Y]$ such that $|V(H') \cap X| \ge \alpha(H)$;
\item\label{item:3.4.ii} there exists a blue copy $H'$ of $H$ in $G[X\cup Y]$ such that $|V(H') \cap Y| \ge \alpha(H)$.
\end{enumerate}
\end{definition}

Analogously to the triangle case, we show that, in the range of $p$ that we are interested in, the random graph $\mathbb{G}(n,p)$ is $(H,s)$-rich with high probability for $s$ linear in $n$. This will also be crucial in identifying an $(H,\eta)$-cluster.

\begin{theorem}\label{thm:Htrans}
Let $H$ be a graph. For every $\varepsilon > 0$, there exists $C>0$ such that the following holds. If $p\geq Cn^{-1/\max\{m_2(H),1\}}$ and $G\sim \mathbb{G}(n,p)$, then with high probability $G$ is $(H,\varepsilon n)$-rich.
\end{theorem}

We discuss \Cref{thm:Htrans} in detail in \Cref{subsec:containers}. For now, we outline how to finish the proof of \Cref{thm:Hlower}. As in the triangle case, we first set aside, for a suitable choice of $\eta$, a maximal family $\mathcal{T}$ of disjoint $(H,\eta)$-clusters. If this family spans a large proportion of the vertices, we obtain a large enough $H$-tiling in both colours. If this is not the case, we iteratively apply the following generalization of \Cref{lem:trianglesup}, thereby obtaining two families $\cB$ and $\cR$ of disjoint blue and red copies of $H$, respectively.

\begin{lemma}\label{lem:Hsup}
Let $H$ be a graph. For every $\varepsilon >0$, there exists $C=C(H,\varepsilon)>0$ such that the following holds. If $p\geq Cn^{-1/\max\{m_2(H),1\}}$ and $G\sim \mathbb{G}(n,p)$, then, with high probability, for every $2$-colouring of the edges of $G$ and every $U\subset V(G)$ with $|U|\geq \varepsilon n$, there exists a monochromatic copy of $H$ in $G[U]$. 
\end{lemma}

Before proceeding, we can easily derive \Cref{lem:Hsup} from \Cref{thm:Htrans}. We emphasize that \Cref{lem:Hsup} can also be derived from the results of R\"{o}dl and Ruci\'{n}ski \cite{RR93, RR94, RR95}.

\begin{proof}[Proof of \Cref{lem:Hsup} assuming \Cref{thm:Htrans}]
By \Cref{thm:Htrans}, there exists a constant $C>0$ such that, for \mbox{$p\ge Cn^{-1/\max\{m_2(H),1\}}$}, with high probability $G$ is $(H,\varepsilon n/2)$-rich. By the definition of richness, it follows that any partition $X \cup Y$ of an arbitrary $\varepsilon n$-subset of $U$ with $|X|=|Y|=\varepsilon n/2$ yields a monochromatic copy of $H$ in $G[U]$.
\end{proof}

Following the steps of the triangle case, if the families $\cB$ and $\cR$ obtained above, together with the $(H,\eta)$-clusters in the family $\mathcal{T}$ we set aside, provide sufficiently many disjoint copies of $H$ in a single colour, then we obtain the desired $H$-tiling. Otherwise, we may assume that both $\cB$ and $\cR$ span sets of vertices of size $\Theta(n)$. The final step is to show that this cannot happen, since in this case we can use the $(H,s)$-richness property and the structure of $\cB$ and $\cR$ to construct a new $(H,\eta)$-cluster, thereby contradicting the maximality of the previously selected family of disjoint clusters.

\subsection{Constructing an $(H,\eta)$-cluster}\label{subsec:2.3} Unlike the triangle case, the difficulty here is that the $(H,s)$-richness property does not directly yield a new $(H,\eta)$-cluster. We briefly sketch how to construct such an $(H,\eta)$-cluster under the assumption of that property. For simplicity, we assume $\eta = 0$ and that each of the families $\cB$ and $\cR$ above has the same size and spans $2m$ vertices.

A cluster can be viewed as a $2$-coloured graph $F$ that contains both a red and a blue $H$-tiling of size approximately $v(F)/(2k-\alpha)$. To construct such a set, let $B$ and $R$ be the vertex sets spanned by $\cB$ and $\cR$, respectively. We then reserve subsets $B_2 \subset B$ and $R_2 \subset R$, each consisting of the vertex set spanned by half of the copies of $H$ in $\cB$ and $\cR$, respectively. These reserved copies will later ensure that the red and blue $H$-tilings in $F$ have the correct size.

Next, we apply the $(H,s)$-richness property using the $m$-sets $B_1 = B \setminus B_2$ and $R_1 = R \setminus R_2$ to construct either a red copy of $H$ intersecting $B_1$ in at least $\alpha(H)$ vertices or a blue copy of $H$ intersecting $R_1$ in at least $\alpha(H)$ vertices. If $m$ is sufficiently large compared to $s$, this procedure can be repeated $(m-s)/k$ times, producing $(m-s)/k$ disjoint monochromatic copies of $H$. For simplicity, suppose that the process stops because fewer than $s$ unused vertices remain in $B_1$. For the sake of brevity, we only analyze here the case where $(m-s)/(2k)$ of these copies are red and each intersects $B_1$ in at least $\alpha(H)$ vertices.

Let $Y$ be the set of vertices defined by the union of these $(m-s)/(2k)$ red copies and let $T' = B_1 \cup Y$. Observe that $G[T']$ contains both a blue $H$-tiling of size $m/k$, since $B_1 \subset T'$, and a red $H$-tiling of size $(m-s)/(2k)$, since $Y \subset T'$. To balance the sizes of the tilings, we now use the reserved set $R_2$. Precisely, we choose a collection of $(m+s)/(2k)$ red copies of $H$ whose vertex set is contained in $R_2$. Let $Z$ be the set of vertices spanned by these copies and define $T = T' \cup Z$. Since $T'$ is disjoint from $R_2$, $G[T]$ contains both a red and a blue $H$-tiling of size $m/k$. Finally, we claim that
\[
|T| \le s + m \le (2k - \alpha)\frac{m}{k}.
\]
To see the first inequality is true, observe that $Y \cup Z$ is the vertex set of an $H$-tiling of size $m/k$ and at most $s$ vertices from $B_1$ are not used to construct this tiling, since our process stopped because fewer than $s$ unused vertices remained in $B_1$. The second inequality follows from the assumptions that $2k-\alpha >k$ and that $m$ is sufficiently large compared to $s$. These observations show that $G[T]$ is an $(H,0)$-cluster.

The following lemma forms the core of the proof of \Cref{thm:Hlower} and synthesizes the conclusions of the preceding analysis.

\begin{restatable}{lemma}{generalHtie}
\label{lem:generalHtie}
Let $H$ be a graph on $k \ge 3$ vertices. Let $s$ be a positive integer and let $\eta>0$. If $G$ is an $(H,\eta^2s)$-rich graph, then for every pair of disjoint $s$-sets $X, Y \subset V(G)$ and every $2$-colouring of the edges of $G$ satisfying that
\begin{enumerate}[label=\textnormal{(\roman*)}]
    \item \label{item:3.5.i} $G[X]$ contains a blue $H$-tiling with $s/k$ copies of $H$;
    \item \label{item:3.5.ii} $G[Y]$ contains a red $H$-tiling with $s/k$ copies of $H$,
\end{enumerate}
there exists a vertex set $T \subset X \cup Y$ such that $G[T]$ forms an $(H,\eta)$-cluster.
\end{restatable}

Several additional cases must be considered, and these are analyzed in full detail in the proof of \Cref{lem:generalHtie} given in \Cref{sec:3}. We emphasize that, in our proof, it is crucial to control the size of the tilings, which explains the need for the error term $\eta$ in the definition of a cluster.

Assuming \Cref{lem:generalHtie}, the proof of \Cref{thm:Hlower} is similar to the proof of \Cref{thm:trianglelower}.

\begin{proof}[Proof of \Cref{thm:Hlower} assuming \Cref{thm:Htrans} and \Cref{lem:generalHtie}]

First we remark that the result is trivial when $H$ has no edges. Moreover, if the statement can be proved for $H = 2K_2$, then the case $H = K_2$ follows immediately, since $\max \{m_2(K_2),1\} = \max \{m_2(2K_2),1\} = 1$. Hence we may assume that $k\ge 3$.

Given $\varepsilon>0$ and $\eta=\varepsilon/(2k-\alpha)$, let $C$ be sufficiently large and let $G \sim \mathbb{G}(n,p)$ for some $p \ge Cn^{-1/\max\{m_2(H),1\}}$. By \Cref{thm:Htrans} and \Cref{lem:Hsup}, with high probability $G$ is $(H,\eta^2\varepsilon n)$-rich and for every $2$-colouring of $G$ and every $U\subset V(G)$ with $|U|\ge \varepsilon n$, there exists a monochromatic copy of $H$ in $G[U]$.

For a fixed $2$-colouring of the edges of $G$, let $\mathcal{T}=\{G[T_1], G[T_2],\dots, G[T_\ell]\}$
be a maximal collection of vertex-disjoint $(H,\eta)$-clusters in $G$ and let $T = T_1\cup T_2\cup\dots\cup T_\ell$ be the set of vertices covered by this collection. If $|T|\geq n-2\varepsilon n$, since each $G[T_i]$ is an $(H,\eta)$-cluster, we can find a red and a blue $H$-tiling in $G[T]$, both consisting of at least
\begin{equation*}
\Big(\frac{1}{2k-\alpha}-\eta\Big)|T|
\geq \frac{n}{2k-\alpha}-\varepsilon n
\end{equation*}
copies of $H$, where the inequality holds by our choice of $\eta$ and $k \ge 3$.

Therefore, we may assume that $|T|\leq n-2\varepsilon n$. Iteratively applying \Cref{lem:Hsup}, we find $(n-|T|-\varepsilon n)/k$ disjoint monochromatic copies of $H$ in $V(G)\setminus T$. Furthermore, observe that we may assume that all but at most $\varepsilon n/k$ of these copies of $H$ must have the same colour. Indeed, if there are at least $\varepsilon n/k$ monochromatic copies of $H$ in each colour, since $G$ is $(H,\eta^2\varepsilon n)$-rich, we can apply \Cref{lem:generalHtie} with $s=\varepsilon n$, where $X$ and $Y$ are the (disjoint) vertex sets of $\varepsilon n/k$ blue copies of $H$ and $\varepsilon n/k$ red copies of $H$, respectively. This yields an additional $(H,\eta)$-cluster in $G$, contradicting the maximality of $\mathcal{T}$. Hence we can find a monochromatic $H$-tiling with at least
\begin{equation*}
    \Big(\frac{1}{2k-\alpha}-\eta\Big)|T| + \frac{n-|T|-\varepsilon n}{k} - \frac{\varepsilon n}{k} \ge \frac{n}{2k-\alpha}-\varepsilon n 
\end{equation*}
copies of $H$, where the inequality is true since $|T| \le n-2\varepsilon n$, by our choice of $\eta$, and $k \ge 3$.
\end{proof}

\subsection{A container lemma}\label{subsec:containers} The proof of \Cref{thm:Htrans} relies on the method of hypergraph containers and follows the general framework used in most applications of the method. We first define a structure that will play an important role in the proof.

\begin{restatable}{definition}{ABGoodCopy}
\label{def:ABgood}
Given a partition $V(K_n)=A\cup B$ and a $2$-colouring $\chi\colon E(K_n) \to \{b,r\}$, we say that a $2$-coloured graph $H' \subset K_n$ is an \emph{$(A,B)$-good copy of $H$ in $\chi$} if $H'$ is isomorphic to $H$ and one of the following holds: \begin{enumerate}[label = (\roman*)]
\item\label{item:good1} $H'$ is red and $|V(H') \cap A| \ge \alpha(H)$;
\item\label{item:good2} $H'$ is blue and $|V(H') \cap B| \ge \alpha(H)$.
\end{enumerate}
\end{restatable}
When the $2$-colouring $\chi$ is clear from context, we shall refer to an $(A,B)$-good copy of $H$ in $\chi$ simply as an $(A,B)$-good copy of $H$. Note that this is a very natural object to study given our definition of $(H,s)$-rich graphs. Indeed, to show that a graph $G\sim\mathbb{G}(n,p)$ is $(H,\varepsilon n)$-rich with high probability for the range of $p$ we are interested in, it suffices to bound the probability that there exist disjoint $\varepsilon n$-sets of vertices $X$ and $Y$ and a $2$-colouring $\chi$ of $G'=G[X \cup Y]$ such that there is no $(X,Y)$-good copy in $\chi$ of $H$ in $G'$.

Before proceeding, we introduce two definitions. A \emph{balanced partition} $V(K_n)=A\cup B$ is a partition satisfying $|A|=|B|$. A $[2]$-\emph{coloured graph} $G$ is a graph together with a function $\chi\colon E(G) \to 2^{\{b,r\}}$ that assigns a non-empty subset of $\{b,r\}$ to each edge. Recall that when every edge of $G$ is assigned to one of the sets $\{b\}$ or $\{r\}$, then we simply have a $2$-colouring of the edges of $G$, and we say that the graph $G$ is $2$-coloured. Observe that a \mbox{$[2]$-coloured} graph can be viewed as a subset of $E(K_n) \times \{b,r\}$ by associating $G$ with the set
\[G_{\chi}\coloneqq\big\{(e,c) \colon e \in E(G), c \in \chi(e)\big\} \subset E(K_n)\times\{b,r\}.\]

To bound the above probability, for each balanced partition $V(K_n)=A\cup B$, we establish the following container lemma for $2$-coloured graphs with no $(A,B)$-good copies of $H$. 

\begin{restatable}{lemma}{ABcontainer}
\label{lem:ABcontainer}
For every graph $H$ there exist $\delta>0$ and $C>0$ such that the following holds. For every $n \in \mathbb{N}$ and every balanced partition $V(K_n) = A \cup B$, there exist a collection $\cG$ of $[2]$-coloured graphs on vertex set $V(K_n)$, and a function $h\colon 2^{E(K_n) \times \{r,b\}} \to \cG$ such that:
\begin{enumerate}[label=\textnormal{(\roman*)}]
\item\label{item:ABcontainer1} each $G \in \cG$ has at most $(1-\delta)\binom{n}{2}$ edges;
\item\label{item:ABcontainer2} for every $2$-coloured graph $G$ on vertex set $V(K_n)$ with no $(A,B)$-good copies of $H$, there exists a $2$-coloured subgraph $S \subset G$ satisfying
\[e(S) \le Cn^{2-1/\max\{m_2(H),1\}}\quad\text{and}\quad G \subset h(S).\]
\end{enumerate}
\end{restatable}

The proof of \Cref{lem:ABcontainer} relies on the general container machinery from~\cite{BMS,ST} and is given in \Cref{sec:5}. The deduction of \Cref{thm:Htrans} from \Cref{lem:ABcontainer} is straightforward and follows the general framework of Nenadov and Steger~\cite{NS}.

\begin{proof}[Proof of \Cref{thm:Htrans} assuming \Cref{lem:ABcontainer}] 

For a fixed $\varepsilon > 0$, let $C'=C'(\varepsilon) > 0$ be sufficiently large, and let $G \sim \mathbb{G}(n,p)$, where $p \ge C' n^{-1/\max\{m_2(H),1\}}$. 
For any pair of disjoint sets $X, Y \subset V(G)$, let $\cG_{\text{bad}}(X,Y)$ be the family of all graphs on vertex set $X\cup Y$ that admit a $2$-colouring of its edges with no $(X,Y)$-good copy of $H$.  Crucially, if $G$ fails to be $(H, \varepsilon n)$-rich, then $G[X \cup Y] \in \mathcal{G}_{\text{bad}}(X,Y)$ for some $\varepsilon n$-sets $X,Y \subset V(G)$.

Fix disjoint $\varepsilon n$-sets $X, Y \subset V(G)$ and let $G' = G[X \cup Y]$. Set $N = 2\varepsilon n$ and observe that $G' \sim \mathbb{G}(N,p)$. To prove that $G$ is $(H, \varepsilon n)$-rich with high probability, our goal is to apply \Cref{lem:ABcontainer} to $G'$ and show that the probability that $G'$ belongs to $\cG_{\text{bad}}$ is at most $\exp(-cpN^2)$ for some $c>0$. We then conclude the argument by taking a union bound over all choices of disjoint $\varepsilon n$-sets $X,Y \subseteq V(G)$.

Let $\delta>0$ and $C>0$ be given by \Cref{lem:ABcontainer} applied to the graph $H$. Let $\cG$ be the collection of $[2]$-coloured graphs on vertex set $X \cup Y$ and $h\colon 2^{(X \cup Y)^{(2)} \times \{r,b\}} \to \cG$ be the function, both given by \Cref{lem:ABcontainer}. Then, if $G'\in \cG_{\text{bad}}(X,Y)$, one can fix a $2$-colouring of $G'$ with no $(X,Y)$-good copies of $H$. Setting $\lambda=C/C'$, by \cref{item:ABcontainer1,item:ABcontainer2} in \Cref{lem:ABcontainer}, there exists a $2$-coloured subgraph $S \subset G'$ satisfying
\begin{equation*}
e(S) \le CN^{2-1/\max\{m_2(H),1\}}\le \lambda pN^2,\quad G' \subset h(S),\quad\text{and}\quad e\big(h(S)\big)\leq(1-\delta)\binom{N}{2}.
\end{equation*}

Let $\cS$ be the family of subgraphs $S \subset K_N$ with at most $\lambda pN^2$ edges. The observations above imply that
\begin{equation*}
\Prob(G' \in \cG_{\text{bad}}(X,Y)) \le \sum_{S \in \cS}\Prob(S \subset G' \subset h(S)) 
\leq \sum_{S\in\cS}p^{e(S)}(1-p)^{e(\overline{h(S)})}\leq e^{-\delta p\binom{N}{2}}\sum_{s=1}^{\lambda pN^2}\binom{N^2}{s}p^s.
\end{equation*}
Since $C'$ is sufficiently large, the last term above can be bounded as
\[
\sum_{s=1}^{\lambda pN^2}\binom{N^2}{s}p^s\leq\sum_{s=1}^{\lambda pN^2} \bigg(\frac{epN^2}{s}\bigg)^{s}\leq N^2\bigg(\frac{e}{\lambda}\bigg)^{\lambda pN^2}\le N^2 e^{\delta pN^2/4}. 
\]
We then obtain
\[
\Prob(G' \in \cG_{\text{bad}}(X,Y)) \leq N^2e^{-\delta pN^2/8} \leq e^{-\delta pN^2/16}.
\]

Hence, taking a union bound over all $\varepsilon n$-sets $X,Y \subset V(G)$, we conclude that
\[\Pr\big(\text{$G$ is not $(H,\varepsilon n)$-rich} \big) \le \binom{n}{\varepsilon n}^2e^{-\delta pN^2/16} \le 4^ne^{-\delta p(2\varepsilon n)^2/16} = o(1),\]
since $\max\{m_2(H),1\} \ge 1$, $C'$ is sufficiently large and $p \ge C'n^{-1/\max\{m_2(H),1\}}$.
\end{proof}

\subsection{Supersaturation for $(A,B)$-good copies of $H$} The proof of \Cref{lem:ABcontainer} relies, as usual in most applications of hypergraph containers, on a supersaturation result. Precisely, we prove the following supersaturation result for $(A,B)$-good copies of $H$.

\begin{restatable}{lemma}{supersatforABGoodcopies}
\label{lem:ABsuper}
For every graph $H$ on $k$ vertices, there exists a constant $\eta=\eta(H)>0$ such that the following holds. Let $G\subset K_n$ be a graph on $n$ vertices with
\[e(G) \geq (1-\eta)\binom{n}{2}.\]
Then, for every balanced partition $V(K_n) = A \cup B$ and every $2$-colouring of the edges of $G$, there exist at least $\eta n^k$ $(A,B)$-good copies of $H$ in $G$.
\end{restatable}

To obtain this result, we adapt the approach of Burr, Erd\H{o}s, and Spencer \cite{BES}. We remark that in their process, they show how to find $H$-ties in the complete graph under certain conditions on the colouring. We conclude this section by sketching their method and comparing it with our approach to prove \Cref{lem:ABsuper}. Further details can be found in \Cref{sec:4}.

Let $k=v(H)$, and let $s$ be a sufficiently large integer depending only on $k$. In their proof, Burr, Erd\H{o}s, and Spencer \cite{BES} show by induction on $m$ that $R(K_s,mH)$ is bounded above by a linear function in $m$. Consequently, for a sufficiently large number of vertices, still linear in $m$, we may, by the symmetry of Ramsey numbers, either find monochromatic copies of $K_s$ in both colours or obtain a monochromatic $H$-tiling of the desired size. In the former case, if $X$ and $Y$ are the vertex sets of the two monochromatic copies of $K_s$ in each colour, to find an $H$-tie, it suffices to find a monochromatic $K_{k,k}$ in the complete bipartite graph $G[X,Y]$. Indeed, suppose without loss of generality that the copy of $K_s$ in $X$ is blue, and that both the copy of $K_s$ in $Y$ and the copy $C$ of $K_{k,k}$ in $G[X,Y]$ are red. The $H$-tie we can find contains a blue copy of $H$ lying entirely in $V(C)\cap X$ while its red copy of $H$ shares an independent set of size $\alpha =\alpha(H)$ with its blue copy of $H$ and has all remaining vertices lying in $V(C)\cap Y$. The choice of $s$ ensures the existence of such a monochromatic complete bipartite graph.

To prove \Cref{lem:ABsuper} we must find several $(A,B)$-good copies of $H$ in $G$ in each $2$-colouring of the edges of $G$. That is, we want many red copies of $H$ each intersecting $A$ in at least $\alpha$ vertices, or many blue copies of $H$ each intersecting $B$ in at least $\alpha$ vertices.

In \Cref{lem:ABsuper}, the graph $G$ is not complete, but a very dense subgraph of the complete graph. In this case, we proceed as follows. For $R = R(k,s)$, classical supersaturation results imply that we have $\Omega(n^R)$ copies of $K_R$ in both $G[A]$ and $G[B]$. In every $2$-colouring of the edges of $G$, each of the $R$-cliques in $A$ then yields a red copy of $K_k$ or a blue copy of $K_s$. The former scenario gives $(A,B)$-good copies of $H$ for free and hence cannot occur too many times, otherwise we have the desired conclusion. In the latter case, a standard counting argument then allows us to assume that there are $\Omega(n^s)$ blue copies of $K_s$ in $A$. Similarly, we can assume that there are $\Omega(n^s)$ red copies of $K_s$ in $B$. The final step is to show that there are $\Omega(n^{2s})$ pairs $(U,V)$, where $U$ is the vertex set of a blue copy of $K_s$ in $A$ and $V$ is the vertex set of a red copy of $K_s$ in $B$, such that the bipartite graph $G[U,V]$ is sufficiently dense to guarantee a monochromatic copy of $K_{k,k}$ as a subgraph. All such monochromatic complete bipartite graphs yield $(A,B)$-good copies of $H$.

\section{Proof of \Cref{lem:generalHtie}}\label{sec:3} 

The purpose of this section is to formalize the outline given in \Cref{sec:2} and prove \Cref{lem:generalHtie}, restated below.

\generalHtie*

We briefly recall the ideas in the proof discussed in \Cref{subsec:2.3}. The first step is to partition $X=X_1 \cup X_2$ and $Y=Y_1 \cup Y_2$ such that each $X_i$ has a blue $H$-tiling of size $s/(2k)$ and each $Y_i$ has a red $H$-tiling of size $s/(2k)$. Then we greedily select many disjoint monochromatic copies of $H$ in $X_1 \cup Y_1$ in each of the two colours using the fact that $G$ is $(H,\eta^2s)$-rich. If among these new copies of $H$ and the ones given on the $H$-tilings of $X_1$ and $Y_1$ we can already find an $(H,\eta)$-cluster, we are done. Otherwise, we use some of the stored copies of $H$ in $X_2$ and $Y_2$ to complete an $(H,\eta)$-cluster.

We can now proceed to the proof of \Cref{lem:generalHtie}.

\begin{proof}[Proof of \Cref{lem:generalHtie}] If $\eta \ge 1$, then $1/(2k-\alpha)<\eta$, since $k\ge 3$ and $k \ge \alpha$. Therefore, in this case any nonempty $2$-coloured graph is an $(H,\eta)$-cluster. Thus we can assume $\eta<1$. Fix a $2$-colouring of the edges of $G$ and let $X,Y \subset V(G)$ be $s$-sets satisfying conditions \ref{item:3.5.i} and \ref{item:3.5.ii} in \Cref{lem:generalHtie}. Then, fix partitions $X=X_1 \cup X_2$ and $Y=Y_1 \cup Y_2$, with $|X_1|=|X_2|=|Y_1|=|Y_2|=s/2$ such that each $X_i$ has a blue $H$-tiling of size $s/(2k)$ and each $Y_i$ has a red $H$-tiling of size $s/(2k)$. Set $A_X = X_1$, $A_Y = Y_1$, $X^{*} = \emptyset$, $Y^{*} = \emptyset$ and $t_X = t_Y= 0$. We run the following process while $\min\{|A_X|,|A_Y|\}\ge \eta^2 s$:
\begin{enumerate}[topsep=6pt, partopsep=6pt, itemsep=4pt, parsep=4pt, label=\arabic*.]
    \item Arbitrarily choose $W \subset A_X$ and $U \subset A_Y$ with 
    $|W| = |U| = \eta^2 s$.

    \item If $G[W \cup U]$ contains a red copy $H'$ of $H$ such that $|V(H') \cap W| \ge \alpha$, then update
    \begin{alignat*}{3}
        A_X &\to A_X \setminus V(H'), \qquad &A_Y &\to A_Y \setminus V(H'), \\
        Y^{*} &\to Y^{*} \cup (V(H') \cap U),\qquad &t_X &\to t_X + 1,
    \end{alignat*}
    and go to Step~1.

    \item If $G[W \cup U]$ contains a blue copy $H'$ of $H$ such that $|V(H') \cap U| \ge \alpha$, then update
    \begin{alignat*}{3}
        A_X &\to A_X \setminus V(H'), \qquad &A_Y &\to A_Y \setminus V(H'), \\
        X^{*} &\to X^{*} \cup (V(H') \cap W),\qquad &t_Y &\to t_Y + 1,
    \end{alignat*}
    and go to Step~1.
\end{enumerate}

When the process stops, return the final tuple
\[(A_X, A_Y, X^*, Y^*, t_X, t_Y).\]

Since $G$ is $(H,\eta^2s)$-rich, while the process does not stop, one of the conditional steps $2$ and $3$ always runs.

Notice that at each step of type 2 that runs, the total number of vertices removed from $A_X \cup A_Y$ is exactly $k$. Moreover, $Y^*$ is precisely the set of vertices removed from $A_Y$ in all steps of type 2. Hence, considering all steps of type 2, the number of vertices removed from $A_X$ is exactly $kt_X-|Y^*|$. Furthermore, at each type 3 step, the number of vertices removed from $A_X$ is at most $k-\alpha$. 
Finally, notice that, without loss of generality, we can assume that when the process stops, we have $|A_X| < \eta^2 s$. The following inequality then holds:
\begin{equation}\label{eq:sumoftiX}
    kt_X - |Y^*| +  (k-\alpha)t_Y  \ge |X_1| - |A_X| > \frac{s}{2} - \eta^2s.
\end{equation}

We consider some cases according to the values of $t_X$ and $t_Y$. In each of these cases, we find a corresponding $(H,\eta)$-cluster.

If $t_X \ge s/(2k)$, let $\widetilde{Y}$ be the set of vertices added to $Y^*$ in the first $s/(2k)$ instances that step~2 ran in the process above. Define $T = X_1 \cup \widetilde{Y}$ and notice that
\[|T| \le |X_1| + \frac{s}{2k}(k-\alpha) = \frac{s}{2k}  (2k-\alpha).\]
Next, note that $G[T]$ contains both a red and a blue $H$-tiling consisting of $s / (2k)$ copies of $H$. Hence, by definition $G[T]$ forms an $(H,0)$-cluster.

If $t_Y \ge s/(2k)$, let $\widetilde{X}$ be the set of vertices added to $X^*$ in the first $s/(2k)$ instances that step~3 ran in the process above. Define $T = Y_1 \cup \widetilde{X}$ and similarly to the previous case, $G[T]$ forms an $(H,0)$-cluster.

Otherwise, we have that $t_X,t_Y < s/(2k)$. Let $Z\subset Y_2$ be the vertex set of  $s/(2k)-t_X$ disjoint red copies of $H$ in $Y_2$. Define $T = X_1 \cup Y^* \cup Z$ and notice that
\[
|T| = \frac{s}{2} + |Y^*| + \Big(\frac{s}{2k} - t_X\Big)k = s + |Y^*| - k t_X.
\]
Thus, by \eqref{eq:sumoftiX} and by the fact that $k \ge \alpha$ and $t_Y < s / (2k)$, we have 
\begin{equation}\label{eq:sizeOfT}
|T| = s + |Y^*| - kt_X < (k-\alpha)t_Y + \frac{s}{2} + \eta^2s \le s - \frac{\alpha s}{2k} + \eta^2s.
\end{equation}

In this case, observe that $G[T]$ contains a blue $H$-tiling with $s/(2k)$ copies of $H$, since $X_1 \subset T$. At the same time, it contains a red $H$-tiling with
\[
t_X + \Big(\frac{s}{2k} - t_X\Big) = \frac{s}{2k}
\]
copies of $H$, those whose existence induced each step $2$ and those originating from $Z$. Therefore, by \eqref{eq:sizeOfT}, the ratio between the size of a maximal monochromatic $H$-tiling in each colour in $G[T]$ and $|T|$ is at least
\[
\frac{s / (2k)}{|T|}
> \frac{s / (2k)}{s - \frac{\alpha s}{2k} + \eta^2s} = \frac{1}{2k - \alpha + 2k\eta^2}.
\]
Finally, since $0< \eta < 1$ and $k \ge 3$, we have
\[\frac{1}{2k - \alpha + 2k\eta^2}
\ge \frac{1}{2k - \alpha} - \eta,
\]
which implies that $G[T]$ is an $(H,\eta)$-cluster.
\end{proof}

\section{Proof of \Cref{lem:ABsuper}}\label{sec:4}

In this section, we prove \Cref{lem:ABsuper}, the supersaturation result needed for the proof of \Cref{lem:ABcontainer}, as discussed in \Cref{sec:2}. We begin by restating \Cref{lem:ABsuper}.

\supersatforABGoodcopies*

Before proceeding, let us outline our strategy. Fix a partition $V(K_n) = A \cup B$ and a $2$-colouring of the edges of $G \subset K_n$. Our proof proceeds in two steps. First, we find $\Omega(n^{R})$ copies of $K_{R}$ in both $G[A]$ and $G[B]$, where $R = R(k,s)$ is the Ramsey number of $K_k$ and $K_s$ for a suitable choice of $s$. To achieve this, we invoke the following classical supersaturation result due to Lov\'asz and Simonovits~\cite{Lovasz-Simonovits}.

\begin{theorem}[Lov\'asz and Simonovits \cite{Lovasz-Simonovits}]\label{thm:lovasz-simonovits} Let $t \ge R\ge 3$ be integers and let $G$ be a graph on $n$ vertices satisfying
\[e(G)\ge \bigg(1 -\frac{1}{t}\bigg)\frac{n^2}{2}.\]
Then $G$ contains at least \[\binom{t}{R}\bigg(\frac{n}{t}\bigg)^R\] 
copies of $K_R$.    
\end{theorem}

In particular, the abundance of $K_{R}$ ensures that either $G[A]$ contains many red copies of $H$ (type~\ref{item:good1}) or many blue copies of $K_{s}$; symmetrically, either $G[B]$ contains many blue copies of $H$ (type~\ref{item:good2}) or many red copies of $K_{s}$. In the next step, if there are not enough copies of $H$ of types~\ref{item:good1} or~\ref{item:good2}, then we find many pairs of copies of $K_{s}$, with one copy in $G[A]$ and the other in $G[B]$, having many edges between them. We then invoke the following classical result due to K\H{o}vari, S\'os, and Tur\'an~\cite{KST}.

\begin{theorem}[K\H{o}vari, S\'os, and Tur\'an \cite{KST}] \label{thm:KST} Let $t_1\le t_2$ be positive integers. Then
    \[\operatorname{ex}(n,K_{t_1,t_2}) = O\big(n^{2-1/t_1}\big).\]
\end{theorem}

This allows us to find a complete bipartite graph between two copies of $K_s$, one in each part. We can then use this complete bipartite graph to find an $(A,B)$-good copy of $H$ of type~\ref{item:good1} or~\ref{item:good2}.

With the above analysis in place, we are now ready to prove \Cref{lem:ABsuper}.

\begin{proof}[Proof of \Cref{lem:ABsuper}]
Let $s$ be an integer such that
\begin{equation}\label{eq:BoundOnS}
2s - k \ge 3\quad\text{and}\quad s^2 > 4\operatorname{ex}(2s,K_{k,k}).
\end{equation}
This is possible by \Cref{thm:KST} and since $s^2 \gg (2s)^{2-1/k}$. Then let $R=R(k,s)$ be the Ramsey number of $K_k$ and $K_s$ and choose $\eta$ satisfying
\[0<\eta < \frac{1}{4(2R)^{2R}}.\]

Given a fixed balanced partition $V(K_n)=A \cup B$ and a $2$-colouring $\chi$ of the edges of $G$, our goal is to find $\eta n^{k}$ $(A,B)$-good copies of $H$ in $\chi$. Suppose that there are fewer than $\eta n^{k}$ red copies of $H$ in $G[A]$ and that there are fewer than $\eta n^{k}$ blue copies of $H$ in $G[B]$; otherwise, we are done. In this setting, we have the following claim. 

\begin{claim}\label{claim:Kscopies}
$G[A]$ contains at least $\eta^{1/2} n^s$ blue copies of $K_s$ and $G[B]$ contains at least $\eta^{1/2} n^s$ red copies of $K_s$.
\end{claim}

\begin{proof} It suffices to show the first statement, as the other follows from symmetry. 
Notice that
\[e(G[A]) \ge (1/2 - 5\eta)|A|^2,\]
since $e(\overline{G}[A]) \le e(\overline{G}) \le \eta n^2$, $|A|=n/2$ and $n$ is sufficiently large.

Our choice of $\eta$ allows us to apply \Cref{thm:lovasz-simonovits} with $t = 1/(10\eta) \ge R$. This shows us that there are at least 
\[\binom{t}{R}\bigg(\frac{n}{2t}\bigg)^R \ge \bigg(\frac{t}{R}\bigg)^R\bigg(\frac{n}{2t}\bigg)^R = (2R)^{-R}n^R\]
copies of $K_R$ in $G[A]$. By the choice of $R$, each of these copies of $K_R$ contains either a red copy of $K_k$ or a blue copy of $K_s$. Recall that $G[A]$ contains at most $\eta n^{k}$ red copies of $H$. Hence, $G[A]$ contains at most $\eta n^{k}$ red copies of $K_k$. Suppose, for the sake of contradiction, that $G[A]$ also contains at most $\eta^{1/2} n^s$ blue copies of $K_s$. Then, a counting argument implies
\[(2R)^{-R} n^R \le\eta n^k\binom{n/2-k}{R-k} + \eta^{1/2} n^s\binom{n/2-s}{R-s} \le \eta n^kn^{R-k} + \eta^{1/2} n^sn^{R-s} < 2\eta^{1/2} n^R,\]
which is false by our choice of $\eta$.
\end{proof}

The goal now is to use these monochromatic copies of $K_s$ in each part to extract many $(A,B)$-good copies of $H$ of types~\ref{item:good1} or~
\ref{item:good2}. Let $\cA \subset 2^{V(A)}$ denote the family of vertex sets of blue copies of $K_s$ in $G[A]$, and let $\cB \subset 2^{V(B)}$ denote the family of vertex sets of red copies of $K_s$ in $G[B]$. In what follows we will find many pairs $(U,W)$ with $U \in \cA$ and $W \in \cB$ with many edges inside $G[U,W]$ in order to apply \Cref{thm:KST} and find a monochromatic copy of $K_{k,k} \subset G[U,W]$.

\begin{claim}\label{claim:UW}
There is a colour $c$ for which there are at least $\eta n^{2s}/4$ pairs $(U,W)$ with $U \in \cA$ and $W \in \cB$ such that $G[U,W]$ contains a monochromatic $K_{k,k}$ in colour $c$.  
\end{claim}

\begin{proof}
It suffices to find at least $\eta n^{2s}/2$ pairs $(U,W)$ with $U \in \cA$ and $W \in \cB$ such that $G[U,W]$ contains a monochromatic $K_{k,k}$. Therefore, by our choice of $s$ satisfying \eqref{eq:BoundOnS}, it suffices to find at least $\eta n^{2s}/2$ pairs $(U,W)$ with $U \in \cA$ and $W \in \cB$ such that $e(G[U,W]) \ge s^2/2$. 

If this is not the case, \Cref{claim:Kscopies} implies that there are at least $\eta n^{2s}$ pairs in $\cA \times \cB$, and thus we have at least $\eta n^{2s}/2$ pairs $(U,W)$ such that $G[U,W]$ contains at least $s^{2}/2$ non-edges. Thus, a double counting shows that
\[\frac{s^2}{2}\frac{\eta n^{2s}}{2} \le e(\overline{G}[A,B])\binom{n/2-1}{s-1}^2 \le e(\overline{G})\binom{n/2-1}{s-1}^2\le \eta n^2\Big(\frac{n}{2}\Big)^{2s-2}=4^{-s+1}\eta n^{2s},\]
where the last inequality is true since $e(\overline{G}[A,B])\le e(\overline{G})\le \eta n^2$. This is a contradiction, since $s^2 > 4^{-s+2}$ for $s\ge 2$.
\end{proof}

Let $c$ be the colour obtained in \Cref{claim:UW}. Suppose $c$ is red. Consider all pairs $(U,W) \in \cA \times \cB$ that contain some copy of $K_{k,k}$ in red. Each of these pairs $(U,W)$ yields an $(A,B)$-good copy of $H$ of type~\ref{item:good1}. Indeed, if $G[U',W']$ is a red copy of $K_{k,k}$ contained in $G[U,W]$ with $U'\subseteq U$ and $W'\subseteq W$, then $G[U'\cup W']$ contains a red copy of $H$ where $V(H)\cap A \subset U'$ is an independent set in $H$ of size $\alpha$.

Therefore, by \Cref{claim:UW} and the observation above, at least $\eta n^{2s}/4$ pairs $(U,W)\in \cA \times \cB$ contain an $(A,B)$-good copy of $H$ of type~\ref{item:good1}. 
On the other hand, each copy of $H$ is contained in at most $\binom{n-k}{2s-k}$ pairs $(U,W) \in \cA \times \cB$. Therefore, assuming that the number of $(A,B)$-good copies of $H$ of type~\ref{item:good1} is at most $\eta n^k$, since otherwise we are done, a double counting argument shows that
\[\frac{\eta n^{2s}}{4} \le \eta n^k \binom{n-k}{2s-k} \le \frac{\eta n^{2s}}{
(2s-k)!},\]
which is a contradiction, since we chose $s$ satisfying \eqref{eq:BoundOnS}. We can similarly deal with the case where $c$ is blue by considering $(A,B)$-good copies of $H$ of type~\ref{item:good2}.
\end{proof}

\section{Proof of \Cref{lem:ABcontainer}}\label{sec:5}

This section is devoted to the proof of \Cref{lem:ABcontainer}. First let us recall its statement below.

\ABcontainer*

As in most applications of the hypergraph container method, we begin by defining the auxiliary hypergraph to which we will apply the container machinery from~\cite{BMS,ST}. Finally, we construct the containers and complete the proof of \Cref{lem:ABcontainer}.

\subsection{The auxiliary hypergraph}\label{subsec:AH} Given a partition $V(K_n)=A \cup B$, we call a graph $G \subset K_n$ \emph{bad} if there exists a $2$-colouring of the edges of $G$ with no $(A,B)$-good copy of $H$ in $G$. Fix a balanced partition $V(K_n)=A\cup B$ and let $\Gamma$ be the set of all $2$-colourings of $K_n$. Given a graph $F \subset K_n$ and a $2$-colouring $\chi\colon E(K_n) \to \{b,r\}$, let 
\[F_{\chi}\coloneqq\big\{(e,\chi(e)) \colon e \in E(F)\big\} \subset E(K_n)\times\{b,r\}.\]
In particular, any $(A,B)$-good copy of $H$ in $\chi$ can be viewed as a subset of $E(K_n)\times\{b,r\}$.

Let $\ell=e(H)$ and consider the following $\ell$-uniform hypergraph $\cH=\cH(A,B,H)$ encoding all $(A,B)$-good copies of $H$ in $\chi$ for each $\chi \in \Gamma$. Precisely, we define
\[V(\cH)\coloneqq E(K_n)\times\{b,r\},\] and $E(\cH)$ to be the collection of all $(A,B)$-good copies of $H$ of some $2$-colouring of $K_n$, that is
\begin{align*}
E(\cH)\coloneqq \Big\{
L_{\chi}\in\binom{V(\cH)}{\ell} \colon \chi\in \Gamma \text{ and } L\text{ is an }  (A,B)\text{-good copy of } H\text{ in }\chi   
\Big\}.   
\end{align*} 

We note, however, that an arbitrary subset $W \subset V(\cH)$ does not necessarily correspond to a coloured graph: for instance, $W$ might contain both $(e,r)$ and $(e,b)$ for some edge $e$, whereas in a coloured graph each edge can be assigned only a single colour. Moreover, given a bad graph $G \subset K_n$ and a $2$-colouring $\chi$ of $G$ for which $G$ contains no $(A,B)$-good copy of $H$, $G_\chi$ is an independent set in $\cH$. Also note that the number of edges in $\cH$ satisfies $e(\cH)\leq 2^{e(H)}n^{k}$.

Before proceeding, let us provide one further definition. Given a subset $W \subseteq V(\cH)$, we define its \textit{shadow graph} to be the graph $\widetilde{W}$ with
\[V\big(\widetilde{W}\big)=V(K_n)\quad\text{ and }\quad E\big(\widetilde{W}\big)=\big\{e\colon \{(e,b),(e,r)\} \cap W \neq \emptyset\big\}.\]
In other words, since a vertex set $W$ is a set of pairs in $E(K_n) \times \{b,r\}$, its shadow graph $\widetilde{W}$ is just the graph obtained by projecting the vertices of $W$ onto their first coordinates.

The idea now is to apply a hypergraph container theorem to $\cH$, in order to derive \Cref{lem:ABcontainer}.

\subsection{Hypergraph containers}\label{subsec:HC} Now we present the hypergraph container theorem that we will use. Prior to that, we introduce a definition from \cite{BMS}.

\begin{definition}
Let $\mathcal{H}$ be a uniform hypergraph with vertex set $V$, 
let $\mathcal{F}$ be an increasing family of subsets of $V$ 
and let $\varepsilon \in (0,1]$. 
We say that $\mathcal{H}$ is $(\mathcal{F}, \varepsilon)$\emph{-dense} if
\[e(\mathcal{H}[A]) \ge \varepsilon e(\mathcal{H})\]
for every $A \in \mathcal{F}$.
\end{definition}

The following container theorem is Theorem 2.2 from \cite{BMS}.

\begin{theorem}[Balogh, Morris, and Samotij \cite{BMS}]\label{thm:BMS}
For every $\ell \in \mathbb{N}$ and all positive $c$ and $\varepsilon$, there exists a positive constant $\gamma$ such that the following holds. 
Let $\mathcal{H}$ be an $\ell$-uniform hypergraph and let $\mathcal{F} \subseteq 2^{V(\mathcal{H})}$ be an increasing family of sets such that 
$|A| \ge \varepsilon v(\mathcal{H})$ for all $A \in \mathcal{F}$. 
Suppose that $\mathcal{H}$ is $(\mathcal{F}, \varepsilon)$-dense and $\tau \in (0,1)$ is such that, for every $j \in [\ell]$,
\[\Delta_j(\mathcal{H}) \le c\tau^{j-1} \frac{e(\mathcal{H})}{v(\mathcal{H})}.\]
Then there exists a family $\mathcal{S} \subseteq \binom{V(\mathcal{H})}{\le \gamma \tau v(\mathcal{H})}$ and functions $f \colon \mathcal{S} \to \overline{\mathcal{F}}$ and 
$g \colon \mathcal{I}(\mathcal{H}) \to \mathcal{S}$ such that for every $I \in \mathcal{I}(\mathcal{H})$,
\[g(I) \subseteq I \quad \text{and} \quad I \setminus g(I) \subseteq f(g(I)).\]
\end{theorem}

To apply \Cref{thm:BMS}, we first construct an increasing family $\cF$ whose members are all large and span many hyperedges. We then show that the hypergraph $\cH$ is $(\cF,\varepsilon)$-dense for some $\varepsilon>0$. This relies on a suitable supersaturation result, which in our setting is provided by \Cref{lem:ABsuper}.

\subsection{Constructing the containers} In this subsection, we construct the containers for $2$-coloured graphs with no $(A,B)$-good copies of $H$ and prove \Cref{lem:ABcontainer} via an application of \Cref{thm:BMS}. Recall that we fixed a balanced partition of $V(K_n)=A\cup B$. 

First, let us verify that the hypergraph $\cH$, defined at the beginning of this section, is $(\cF, \varepsilon)$-dense for a suitable family $\cF$ and some~$\varepsilon>0$.

\begin{lemma}\label{lem:dense}
There exists $\varepsilon>0$ such that the following holds. Let $\cF \subset 2^{V(\cH)}$ be defined as
\[\cF \coloneqq \Big\{W \subset V(\cH) \colon e\big(\widetilde{W}\big) \ge \big(1 - 2^{e(H)}\varepsilon\big) \binom{n}{2}\Big\}.\]
Then $\cH$ is $(\cF, \varepsilon)$-dense.
\end{lemma}
\begin{proof}
Let $\eta>0$ be given by \Cref{lem:ABsuper} and let $\varepsilon = \eta/2^{e(H)}$. Note that $\cF$ is increasing. By \Cref{lem:ABsuper}, the definition of $\cF$ implies that $\widetilde{W}$ contains at least $2^{e(H)}\varepsilon n^{k}$ $(A,B)$-good copies of $H$. Therefore, we have
\[e(\cH[W])\geq 2^{e(H)}\varepsilon n^{k}\geq
\varepsilon e(\cH),\]
where the last inequality holds since $e(\cH)\leq 2^{e(H)}n^{k}$, as required.
\end{proof}

To apply \Cref{thm:BMS}, we need the following observation on the degrees of $\cH$.

\begin{lemma}\label{lem:degreecondition}
Let $\tau=n^{-1/\max\{m_2(H),1\}}$. For each $1\leq j\leq e(H)$, the maximum $j$-degree of $\cH$ satisfies
\[\Delta_j(\cH)\leq e(H)!\cdot\tau^{j-1} n^{k-2}. \]
\end{lemma}
\begin{proof}
First, for $j=1$, observe that each edge of $K_n$ lies in at most $e(H) n^{k-2}$ copies of $H$, 
and the same bound applies to $(A,B)$-good copies. Hence, for all $e \in E(K_n)$ we have
\[d_{\cH}(e) \leq e(H) n^{k-2}.\]

Next, for each $2\leq j\leq e(H)$, fix an arbitrary set $U\subset V(\cH)$ with $|U|=j$. 
If $U$ contains both $(e,b)$ and $(e,r)$ for some edge $e$, then no $(A,B)$-good copy of $H$ can contain $U$ and therefore $d_{\cH}(U)=0$. If this is not the case, then $U$ corresponds to a set of edges in $K_n$. 
Let $F$ denote the subgraph of $K_n$ induced by these edges. Observe that there are at most 
\[\frac{e(H)!}{(e(H)-e(F))!} n^{k-v(F)} \le e(H)!\cdot n^{k-v(F)}\] $(A,B)$-good copies of $H$ containing $F$. Thus, by the definition of $m_2(H)$, we conclude that
\[k-v(F)=-\bigg(\frac{v(F)-2}{j-1}\bigg)(j-1)+k-2\le-\frac{(j-1)}{\max\{m_2(H),1\}}+k-2\]
which implies
\begin{equation*}
    d_{\cH}(U) \leq e(H)!\cdot n^{k-v(F)} \leq e(H)!\cdot\tau^{j-1} n^{k-2}
\end{equation*}
for every set $U\subset V(\cH)$ of size $j$, as claimed.
\end{proof}

We are now ready to prove \Cref{lem:ABcontainer}. Throughout the proof, any subset of $V(\cH)$ will be viewed as a $[2]$-coloured graph. Moreover, for a $[2]$-coloured graph $S$, we will refer to $e(\widetilde{S})$ simply as $e(S)$.

\begin{proof}[Proof of \Cref{lem:ABcontainer}]
Recall that $k=v(H)$, and let $\ell=e(H)$ and $c=(2k)^{k}\ell!$. By \Cref{lem:degreecondition}, for each $1\leq j\leq e(H)$, the maximum $j$-degree of $\cH$ satisfies
\[\Delta_j(\cH)\leq \ell!\cdot \tau^{j-1} n^{k-2}\leq c  \tau ^{j-1}\frac{e(\cH)}{v(\cH)},\]
where the last inequality is true since
\[e(\cH)\geq \binom{n/2}{k}\ge \bigg(\frac{n}{2k}\bigg)^{k}\quad\text{and}\quad v(\cH)=2\binom{n}{2}\leq n^2.\]

Let $\eta$ be given by \Cref{lem:ABsuper}, $\cF$ and $\varepsilon$ be given by \Cref{lem:dense} and recall that $\varepsilon=\eta/2^{e(H)}$. We already know that $\cF$ is increasing and \Cref{lem:dense} says that $\cH$ is $(\cF, \varepsilon)$-dense. It remains to verify that $|A| \ge \varepsilon v(\cH)$ for all $A \in \cF$. This follows from
\[|A| \ge e(A) \ge (1 - 2^{e(H)}\varepsilon)\binom{n}{2} \ge \varepsilon v(\cH),\]
where the last inequality is true since $v(\cH) \le n^2$, $\varepsilon=\eta/2^{e(H)}$ and $\eta<1/4$.

Therefore, we can apply \Cref{thm:BMS} to $\cH$, which ensures the existence of a constant $\gamma>0$, a collection $\mathcal{S} \subseteq \binom{V(\mathcal{H})}{\le \gamma\tau\, v(\mathcal{H})}$, and functions
$f\colon \cS \to \overline{\cF}$ and $g \colon \cI(\cH) \to \cS$ such that for every $I \in \cI(\cH)$,
\begin{equation}\label{eq:contIf}
g(I) \subseteq I 
\quad \text{and} \quad 
I \setminus g(I) \subseteq f(g(I)).
\end{equation}

Let
\[\cG\coloneqq \{S \cup f(S) : S \in \cS\},\]
and define $h: 2^{V(\cH)} \to \cG$ by
\[h(S) = S \cup f(S)\]
if $S \in \cS$. For $S \not\in \cS$ define $h(S)$ arbitrarily.

Every $2$-coloured graph $G$ on $n$ vertices with no $(A,B)$-good copies of $H$ can be viewed as an independent set $I_G \in \cI(\cH)$. Let $S = g(I_G)$ and note that \eqref{eq:contIf} implies
\begin{equation}\label{eq:contIS}
    S \subseteq I_G \subseteq S \cup f(S)=h(S).
\end{equation}
Now let $C=\gamma$ and observe that
\begin{equation}\label{eq:sizeofS}
e(S) \le |S| \le \gamma\tau v(\cH) \le C n^{2-1/\max\{m_2(H),1\}},
\end{equation}
which, together with \eqref{eq:contIS}, proves \cref{item:ABcontainer2} in \Cref{lem:ABcontainer}.

Let $\delta = \eta/2$. It remains to prove that for each $X\in\cG$ we have $e\big(X\big)\leq (1-\delta)\binom{n}{2}$. Let $S \in \cS$ and set $X=S \cup f(S)$. By \eqref{eq:sizeofS}, $S$ satisfies
\[e(S)\le |S| \leq Cn^{2-1/\max\{m_2(H),1\}}<\delta \binom{n}{2},\]
where the last inequality holds for sufficiently large $n$. Finally, note that
\[e(X) \le e\big(S\big) + e\big(f(S)\big) \leq \delta\binom{n}{2} + (1-2\delta)\binom{n}{2} = (1-\delta)\binom{n}{2},\]
since $f(S) \in \overline{\cF}$, $\eta = \varepsilon2^{e(H)}$ and $\delta=\eta/2$. This proves \cref{item:ABcontainer1} in \Cref{lem:ABcontainer}.
\end{proof}

\section*{Acknowledgements}
This project was initiated during the $2^a$ Escola Brasileira de Combinatória held at IMPA, Rio de Janeiro. We are very grateful to Rob Morris, Hong Liu, and Ayush Basu for their helpful discussions and suggestions on the presentation of this paper.

\bigskip
\bibliographystyle{abbrv}
\def\bibfont{\footnotesize}
\bibliography{refs}
\end{document}